\documentclass{amsart}
\usepackage{todonotes}
\usepackage{amssymb,amsmath}
\usepackage[stable]{footmisc}
\usepackage{xypic}
\usepackage{mathrsfs}

\def\C{\mathfrak C}

\def\T{\bold T}
\def\K{\bold K}
\def\R{\mathbb R}

\def\emptyset{\varnothing}

\def\L{L}

\def\bigland{\bold{\bigwedge}}
\def\biglor{\bold{\bigvee}}

\def\fraisse{Fra\"{\i}ss\'e}

\def\finsubs{ \subseteq_\omega}
\def\x{\bar x}
\def\y{\bar y}
\def\z{\bar z}

\def\isom{\cong}
\newcommand\rest{\mathbin{\!\upharpoonright\!}}

\def\sleq{\preccurlyeq}

\def\a{{\bar a}}
\def\b{{\bar b}}
\def\c{{\bar c}}
\def\v{\bar v}
\def\w{\bar w} 
\def\y{\bar y} 
\def\a{\bar a}
\def\b{\bar b}
\def\c{\bar c} 
\def\d{\bar d}
\def\isom{\approx}

\def\f{\hat f}
\def\m{\bar m}

\let\cl\undefined

\let\gcl\undefined

\DeclareMathOperator{\cl}{cl}
\DeclareMathOperator{\fn}{Fn}

\DeclareMathOperator{\ecl}{ecl}

\DeclareMathOperator{\tp}{tp}
\DeclareMathOperator{\cltp}{cltp}

\DeclareMathOperator{\supp}{supp}

\newtheorem{thm}{Theorem}[section]

\newtheorem{cor}[thm]{Corollary}
\newtheorem{lem}[thm]{Lemma}

\newtheorem{notn}[thm]{Notation}
\newtheorem{conv}[thm]{Convention}
\newtheorem{prop}[thm]{Proposition}

\newtheorem{ex}[thm]{Example}

\theoremstyle{remark}
\newtheorem*{rem}{Remark}

\theoremstyle{definition}
\newtheorem{defn}[thm]{Definition}

\def\Ind{\setbox0=\hbox{$x$}\kern\wd0\hbox to 0pt{\hss$\mid$\hss} \lower.9\ht0\hbox to 0pt{\hss$\smile$\hss}\kern\wd0}

\def\Notind{\setbox0=\hbox{$x$}\kern\wd0\hbox to 0pt{\mathchardef \nn=12854\hss$\nn$\kern1.4\wd0\hss}\hbox to 0pt{\hss$\mid$\hss}\lower.9\ht0 \hbox to 0pt{\hss$\smile$\hss}\kern\wd0}

\title{Full Amalgamation Classes with Intrinsic Transcendentals}
\author{Justin Brody}

\address{Department of Mathematics and Computer Science\\
Goucher College\\
1021 Dulaney Valley Road\\
Baltimore, MD 21204}
\email{justin.brody@goucher.edu}

\begin{document}
\begin{abstract}
  We develop some basic results about full amalgamation classes with
  {\it intrinsic trascendentals}.  These classes have generics whose
  models may have finite subsets whose intrinsic closure is not
  contained in its algebraic clsoure.  We will show that under fairly
  natural conditions the generic will have an essentially undecidable
  theory, but we will also exhibit strictly superstable and strictly
  simple examples.  Separating types over a model into those that are
  {\it intrinsic} and those that are {\it extrinsic}, we will
  demonstrate that the complexity exceeding that of a simple theory in
  the classes with essentially undecidable theories of
  \cite{brody-laskowski} comes from the intrinsic types by deriving a
  class from them which has a strictly simple theory with few
  intrinsic types.
\end{abstract}
\maketitle

\section{Introduction}
\label{sec:intro}
This paper is concerned with the complexity of generic structures
arising from full amalgamation classes with what we term {\it
  intrinsic transcendentals}.  The latter are elements which will be
contained in the closure of a finite subset of a model of the generic
but will not be algebraic over that set.  Specifically, let
$(\K, \leq)$ be a full amalgamation class with intrinsic
transcendentals and let $T$ be the theory of the $(\K, \leq)$-generic.
We will give a combinatiorial condition under which $T$ is essentially
undecidable; this condition will be met by many classes derived from a
pre-dimension function in the usual way.  We will introduce the notion
of a {\it support system}, which is used to measure the complexity of
the ways certain intrinsic extensions can attach to a model.  We will
show that by limiting that complexity we can tame the model-theoretic
complexity of $T$ to a certain degree.  Specifically, for
$M \models T$ we divide the types in $S(M)$ into those that are {\it
  intrinsic} and those that are {\it extrinsic}.  If the class
$(\K, \leq)$ has limited supports, then the number of intrinsic types
over $M$ will be limited.  While the resulting generic may still be
unstable (due to the number of extrinsic types), we will show that in
a particular family of amalgamation classes the number of and
structure of the extrinsic types will not exceed the conditions of
simplicity.  Thus if a class has a generic with a non-simple theory,
this will be due to the number of intrinsic types over models of that
theory.

The history of the Hrushovski construction is well-detailed in many
sources (for example, see \cite{mcl-kuek, baldwin-shi, wagner}) to
which the reader is referred for a full accounting.  The basic idea,
which we will reprise in Section \ref{sec:background}, is that a class
of finite structures $\K$ which are partially ordered by a notion of
strong substructure $\leq$ can be amalgamated to produce a canonical
generic of the class.  The model theory of the generic is determined
by the partial order $\leq$, and by choosing appropriate notions of
strong substructure Hrushovski was able to refute a number of
then-current conjectures in geometric stability theory
\cite{hrushovski-new-sm,hrushovski-stable-pp,hru-acf}.  The
construction continues to play a prominent role in many ongoing
research programmes.

Every instance of the construction is associated with a closure
operator which is determined by $\K$ and
$\leq$.  For many of the studied examples of the construction
\cite{hrushovski-new-sm, baldwin-shi, baldwin-shelah, mcl} the
relation $\leq$
is chosen so that the closure of finite subset of a model of the
theory the generic would be contained in the algebraic closure of that
set.  Pourmahdian \cite{pourmahdian} called this the {\it algebraic
  closure} property (AC) and made a study of certain classes which did
not possess the property.  He showed that performing a kind of partial
Morleyization on the descriptions of the closures of sets yields a
kind of quantifier eliminiation.  He also showed the simplicity of the
resulting theory.  This paper contains somewhat analagous results; our
introduction of closure types in Section \ref{sec:cltp} provides a similar kind
of quantifier eliminiation that does not require an expansion of the
language.  Our result that reducing the number of intrinsic types
results in a simple structure (Theorem \ref{simple}) is similar in flavor to
Pourmahdian's simplicity result.  On the other hand, we note
that a previous paper \cite{brody-laskowski} and more generally the
results of Section \ref{sec:arith} answer Pourmahdian's question of
whether the first-order theory of the generic need be simple (Question
4.10 of \cite{pourmahdian}) in the negative.

In part, the significance of our results is that the methods described
here can potentially produce well-behaved generics with rich
geometries that can shed light on the nature of amalgamation
constructions.  For example, Baldwin has conjectured that no generic
with a small theory will be strictly superstable.  We will exhibit a
non-small strictly superstable generic in section 5, and the
techniques developed here may ultimately shed some light on the main
conjecture.

The organization of the paper is as follows.  In Section
\ref{sec:background} we outline Hrushovski amalgamation construction,
define intrinsic transcendentals, and give examples of generics which
have them.  Section \ref{sec:cltp} introduces a syntatic device for
counting types when automorphism arguments are not available (this
will be the usual case in classes with intrinsic transcendentals,
since the closure of a model will in general have the same power as
the universal domain in which it lives).  Section \ref{sec:arith} will
show that classes with intrinsic transcendentals often have
essentially undecidable theories.  The following section,
\ref{sec:limiting}, gives a condition under which the number of types
represented in the closure of a model is limited.  In Section
\ref{sec:acoll}, we will employ this condition to derive a generic
with a simple theory from a class whose generic has an essentially
undecidable theory.  The main technique in this section will be to
adapt ideas from Hrushovski's collapse to limit the structure of the
intrinsic closure.

I would like to thank a number of people who have been extremely
helpful in bringing this paper to its current form.  Koichiro Ikeda
read through many drafts of the paper and made extremely useful
suggestions (for example, the distinction between intrinsic and
extrinsic types is a result of his comments).  Chris Laskowski pointed
out the need to develop the mechanism in Section \ref{sec:cltp} and
John Baldwin provided invaluable feedback about the definition of
supports in Section \ref{sec:limiting} and the overall presentation of
the material.  I am extremely grateful for all of their efforts.

\section{Background}
\label{sec:background}

\subsection{Hrushovski's Amalgamation Construction}

Hrushovski's amalgamation construction generalizes the \fraisse\
construction to allow for the production of a more general class of
structures.  In particular, rather than amalgamating a class of finite
structures partially ordered by the substructure relation, the
Hrushovski construction amalgamates a class of finite structures
partially ordered by a {\bf strong} (or {\bf closed}) substructure
relation.  This allows for a finer level of control on the resulting
structure.  The choice of the strong substructure relation forces
various model-theoretic properties on the final (generic) structure.
In particular, the relation determines a closure operator on the
generic which in turn determines much of its model-theory.

We present the construction in an arbitrary finite relational language
even though our main examples in this paper will either be graphs or
unary-predicate expansions of graphs.  For any finite relational
language $\L$, we want to work with a class of finite $\L$-structures
$\K$ ordered by a strong substructure relation $\leq$ so that the pair
$(\K, \leq)$ satisfies the axioms in Definition
\ref{sec:background}.\ref{kax} below.  It is worth noting that
our requirements (especially the assumption of full amalgamation) are
significantly stronger than the minimal set of assumptions needed to
carry out the construction.  The assumption of full amalgamation will
allow us to simplify our analyses, and specifically allow the tools
developed in Section \ref{sec:cltp} to work.

\begin{conv}
  For $A, B$ subsets of a common superstructure, $AB$ denotes the
  union of $A$ and $B$ and $A \finsubs B$ denotes that $A$ is a finite
  substructure of $B$.
\end{conv}

\begin{defn}\label{kax} For any finite relational language $\L$, $\K$ a
  class of finite $\L$-structures and $\leq$ a binary relation on $\K$,
  we say that $(\K, \leq)$ is a {\it full amalgamation class} if the
  following properties are satisfied for any $A,B, C \in \K$:

\begin{enumerate}
\item\label{kax:1} The class $\K$ is closed under substructures and
  isomorphisms.
\item\label{kax:4} The relation $\leq$ is isomorphism-invariant: if $A
  \leq B$ and $f: B \isom B'$, then $f(A) \leq B'$.
\item\label{kax:5} $A \leq A$.
\item\label{kax:3} If $A \leq B$, then $A \subseteq B$.
\item\label{kax:2} $\emptyset \leq A$.
\item\label{kax:6} If $A \leq B$ and $B \leq C$, then $A \leq C$.
\item\label{kax:7} If $A \leq B$ and $A \subseteq C \subseteq B$,
  then $A \leq C$.
\item\label{kax:8} For $A \leq B$ from $\K$, we have $A \cap C \leq B
  \cap C$ for every $C \in \K$.
\item\label{kax:9}  The pair $(\K, \leq)$ has the {\it full amalgamation
    property}, discussed below.  
\end{enumerate}
For $A, B \in \K$, if $A \leq B$ then we will say that $A$ is {\it
  strong} (or {\it closed}) in $B$.
\end{defn}

The amalgamation property says that we can coherently join two strong
extensions of a structure $A$ into a new structure $D \in \K$.  Some
form of amalgamation is central to any form of this construction; it
allows us to canonically join elements of $\K$ to produce a unique
countable structure, the $(\K, \leq)$-{\bf generic}.

Specifically, let us say that a map $f: A \rightarrow B$ is {\it
  strong} if $f(A) \leq B$ for $A,B \in \K$.  Then, given $A,B,C \in
\K$ and $f$ as in the diagram below (where $\leq$ above an arrow
indicates that the corresponding map should be strong, and arrows
labeled $i$ are inclusions ), the amalgamation property states that we
can find $D \in \K$ and a strong map $g$ for which the diagram
commutes.

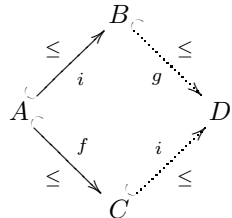
\centerline{ \xymatrix{
    &B \ar@{^{(}.>}_{g}^{\leq}[rd]    \\
    A \ar@{^{(}->}_{i}^{\leq}[ru] \ar@{^{(}->}_{\leq}^{f}[rd] &&D \\
    &C \ar@{^{(}.>}_{\leq}^{i}[ru] }} We will call $D$ an {\it
  amalgam} of $B$ and $C$ over $A$.

We will want to work with classes which have a particularly strong
form of the amalgamation property.  These are the full amalgamation
classes defined by Baldwin and Shi (see \cite{baldwin-shi}).

\begin{defn} Suppose $A,B,C$ are elements of $\K$ with $A = B \cap C$,
  and let $D$ be the structure whose universe is $BC$ and whose
  relations are precisely those of $B$ and those of $C$.  Then we will
  denote $D$ by $B \oplus_{A} C$
  
  \begin{itemize}
  \item If $(\K, \leq)$ is an amalgamation class in which $B
    \oplus_{A} C$ is an amalgam of $B$ and $C$ over $A$, then we will
    call $B \oplus_{A} C$ the {\it free amalgam} of $B$ and $C$ over
    $A$ and say that $(\K, \leq)$ is a {\it free amalgamation class}.

    As notation, if $\{\, B_i : i < n \,\} $ is a family of structures
    with $B_i \cap B_j = A$ for $i \neq j$, then we write
    $\oplus_{i<n}(B_i/A)$ to denote $(\oplus_{i<n-1}(B_i/A)) \oplus_A
    B_{n-1}$ when $n > 2$ ( $\oplus_{i<2}(B_i/A) = B_0 \oplus_A B_1$).

  \item A free amalgamation classes is {\it full} if for $A,B,C\in\K$
    and $A\le B$, $ A\subseteq C$, then $C\le D$, where
    $D=B \oplus_A C$.
  \end{itemize}
\end{defn}


If $N$ is any $\L$-structure whose finite subsets are all in $\K$,
then the closed substructure relation can be extended to arbitrary
subsets of $N$, this in turn determines a closure operator on
$N$.  The definition depends on the notion of a {\it minimal pair}
$(A,B)$:  for $A \finsubs B \in \K$, $(A,B)$ is a minimal pair if $A
\not \leq B$ but for every $A \subseteq B_0 \subsetneq B$, $A \leq
B_0$.  These represent minimal instances of extensions which are not
strong; the intuition is that a (possibly infinite) set will be closed
if it is closed under the operation of extending finite subsets by
minimal pairs.

\begin{defn}
  Let $(\K, \leq)$ be a full amalgamation class and let $N$ be an
  $L$-structure  whose  finite subsets are elements of $\K$.
  \begin{itemize}
  \item If $A \finsubs B \finsubs N$, then $A \leq B$ is determined by
    the $\leq$-relation on $\K$.
  \item If  $A \finsubs M$ (with $M$ possibly infinite),
    then $A \leq M$ exactly when $A \leq B$ for  every $A \subseteq B \finsubs M$.
  \item If $M \subseteq N$, then $M \leq N$ exactly when for
    $A \finsubs M$, if $(A, B)$ is a minimal pair for some
    $B \finsubs N$, then $B \subseteq M$.
  \end{itemize}
\end{defn}

This leads to the  critical notion of a closure:  for $M \subseteq N$
the $N$-closure of $M$, denoted $\cl_N(M)$ is the smallest $M'$ such
that $M \subseteq M' \leq N$.  Our axioms (particularly
\ref{kax}.\ref{kax:8}) will guarantee that this is well-defined.  

The closure operator determines a pre-geometry (without exchange) on
the structures $N$ whose finite substructures are in $\K$; it is largely through determining the properties of the
closure that the Hrushovski construction forces various
model-theoretic properties in the generic, which we now define.  

If $(\K, \leq)$ is a full amalgamation class, then we can imitate the
construction of the \fraisse\ limit (amalgamating over strong
substructures rather than over all substructures) and produce the
$(\K, \leq)$-{\it generic}.  This is a countable $\L$-structure $M$
(unique up to isomorphism) that satisfies the following constitutive
properties.
\begin{enumerate}
\item For $A \finsubs  M$, $A  \in \K$
\item If $A \leq M$ and $A \leq B$, then there is a strong embedding
  of $B$ into $M$ over $A$.
\item For any finite $A \finsubs M$, $\cl_M(A)$ is finite.
\end{enumerate}

We adopt the following conventions.
\begin{conv}\ 
  \begin{itemize}
  \item For $(\K, \leq)$ a full amalgamation class, we will write
    $T_{(\K, \leq)}$ to denote the theory of the $(\K, \leq)$-generic
    and let $\C_{(\K, \leq)}$ denote a universal domain for
    $T_{(\K, \leq)}$.  That is, $\C_{(\K, \leq)}$ is chosen to be a
    $\kappa$-saturated, strongly $\kappa$-homogeneous model of
    $T_{(\K, \leq)}$ for $\kappa$ larger than the size of any model
    we're working with.  If context makes the amalgamation class
    clear, we may simply write $\C$.
  \item We will use $(\K, \leq)$ to denote an arbitrary amalgamation
    class and introduce extra notation (e.g. $(\K_r, \leq_r)$) to
    indicate a specific class.
  \end{itemize}
\end{conv}

\subsection{Intrinsic Transcendentals}

Our main interest in this paper will be in amalgamation classes which
do not have what Pourmahdian \cite{pourmahdian} calls the algebraic closure
property.  Specifically, we will be interested in classes which
satisfy the following condition.
\begin{defn}
  The class $(\K, \leq)$ has {\it intrinsic transcedentals} if there
  is a minimal pair $(A,B)$ from $(\K, \leq)$ such that for any
  $n \in \omega$ there is some $D_n \in \K$ which contains a copy of
  the free amalgam $ \bigoplus_{i < n} B/A$.
\end{defn}

The primary examples of such classes arise from {\it predimension}
functions.  These are functions $\delta:  \K \to \R$ which are used to
define a strong substructure relation $\leq$ on $\K$.  When $\K$ is a
class of graphs, then for any real $\alpha \geq  0$ we will be
particularly interested in the following.

\begin{defn}
  For any graph $A$ and real $\alpha \geq 0$, the {\it predimension of
    $A$}, $\delta_\alpha(A)$ is given by
  $\delta_\alpha(A) = |A| - \alpha e(A)$ where $e(A)$ denotes the
  number of edges in $A$.  For $A \finsubs B$, the {\it relative
    predimension of $B$ over $A$}, is given by
  $\delta_\alpha(B/A) = \delta_\alpha(B) - \delta_\alpha(A)$
\end{defn}

For any such $\alpha$, this gives us two distinct ways to define an
amalgamation class.
\begin{defn}
  Fix a real $\alpha \ge 0$.
  \begin{itemize}
  \item $\K_\alpha$ is the class of all finite graphs which have
    hereditarily non-negative predimension.  That is $\K_\alpha = \{\,
    A : \delta_\alpha(A') \geq 0 \text{ for } A' \subseteq A \,\}$.  For
    $A \finsubs B \in \K_\alpha$, we say that $A \leq_\alpha B$ if for
    $A \subseteq B_0 \subseteq B$, $\delta_\alpha(B_0 / A) \geq 0$.
  \item $\K_\alpha+$ is the class of all finite graphs which have
    hereditarily positive predimension.  That is $\K_\alpha^+ = \{\,
    A : \delta_\alpha(A') > 0 \text{ for } A' \subseteq A, A' \neq \emptyset \,\}$.  For
    $A \finsubs B \in \K_\alpha^+$, we say that $A \sleq_\alpha B$ if for
    $A \subsetneq B_0 \subseteq B$, $\delta_\alpha(B_0 / A) > 0$.
  \end{itemize}
\end{defn}
Note that for irrational $\alpha$ the two classes coincide since
$\delta_\alpha$ is never 0.  For rational $r$, it is shown in
\cite{baldwin-shi} that $T_{(\K_r, \leq_r)}$ is $\omega$-stable while
\cite{brody-laskowski} shows that $T_{(\K_r^+, \sleq_r)}$ (for
$0 < r < 1$) is essentially undecidable.  The primary difference
between these two classes is that in the former class extensions of
relative predimension $0$ are strong and must occur over any finite
set, while in the latter class such extensions can occur or not in
arbitrarily complex configurations.  Such classes form our primary
example of classes with intrinsic transcendentals; in Section
\ref{sec:acoll} we will modify these classes to produce tamer
theories.

\section{The Closure Type}
\label{sec:cltp}
For a given class $(\K, \leq)$, our intuition is that the model theory
of $T_{(\K, \leq)}$ is determined by the complexity of the associated
closure operation.  This is made precise in the $(\K, \leq)$-generic
by noting that isomorphisms of closed subsets extend to automorphisms
of the generic.  This will not extend to the universal domain
$\C_{(\K, \leq)}$ since closed sets will in general have the same
cardinality as $\C$ and we can thus not appeal to strong homogeneity.
We can, however, get a similar result by working with back-and-forth
equivalence rather than automorphisms.  In particular, we will show
that if $X$ and $Y$ are closed subsets of $\C$ which are elementarily
equivalent as $\L$-structures (i.e. $X \equiv Y$) then they are
equivalent as substructures of $\C$ (i.e. $(\C, X) \equiv (\C, Y)$).


Our technical device for showing this is the closure type, which
determines the elementary structure of the closure of a set.
Throughout this section, fix a full amalgamation class $(\K, \leq)$ and
let $\C = \C_{(\K, \leq)}$.  Our main result will be that if finite
sets have the same closure type, then they realize the same complete
type in $\C$.

The closure type is modeled on the game-normal formulae used in the
proof of the \fraisse-Hintikka theorem (see \cite{Hodges}) and makes
fundamental use of a way of decomposing the closure of a set noted by
Baldwin and Shi in \cite{baldwin-shi}.  The latter paper noticed that
the closure of a finite set could be obtained by iteratively adding
minimal pairs to the set.  In particular, they defined
$I_0(\a) = \bigcup \{\, B : (\a_0, B) \text{ a minimal pair}, \a_0
\finsubs \a \,\} $
and
$I_{k+1}(\a) = \bigcup \{\, B : (\a_0, B) \text{ a minimal pair}, \a_0
\finsubs I_{k}(\a) \,\} $.
Then we have $\cl(\a) = \bigcup_{j < \omega} I_j(\a)$: the latter set
is closed because any minimal pair $(A,B)$ with
$A \subseteq \bigcup I_j(\a)$ must have $A \subseteq I_n(\a)$ for some
$n$ and thus $B \subseteq I_{n+1}(\a)$.

Imagine a modified Ehrenfeucht-\fraisse\ game in which two players are
working on the closures of tuples $\a$ and $\b$, and at each round
both players extend the vertices already played by a minimal pair.
Thus in a 1-round game the ``spoiler'' might pick an extension $A$
where $(\a, A)$ is a minimal pair, and it will suffice for the
``duplicator'' to find a copy of $A$ over $\b$.  If playing a 2-round
game the duplicator will need to be more careful about her copy of $A$
-- it will have to be a copy for which she knows that she can respond
appropriately to the next move of the spoiler.  In particular, she
will have to choose $B$ a copy of $A$ for which the minimal pairs over
$\b B$ correspond precisely to those over $\a A$.  If the game will go
for $(k+1)$-rounds, then the duplicator will need to choose a copy of
$A$ which can support all the possible sequences of $k$ moves by the
spoiler.

The idea of an intrinsic formula is to code this information
syntactically.  In particular, a $k$-intrinsic formula over $\a$ will
describe a minimal extension along with the possible combinations of
$k-1$ moves that could be supported by that extension.

Formally, we define a $0$-intrinsic formula over $\a$ to be of
the form
\begin{equation}
  \label{kint-0}
  \phi(\x; \a) := \Delta_B(\a \x)
\end{equation}
where $(\a_0, B)$ is a minimal pair for $\a_0 \subseteq \a$ and
$\Delta_B(\a \x)$ asserts that $\a_0 \x$ is isomorphic to $B$ (note
that the notation is ambiguous -- $\Delta_B(\a \x)$ could refer to
several different formulae depending on which subset of $\a$ is picked
out.  This should not cause any problems in what follows.)

Having defined $k$-intrinic formulae, we define a $k+1$-intrinsic
formula to be of the form
\begin{equation}
  \label{kint-ind}
  \phi(\x; \a) := \Delta_B(\a \x) \land \bigland_{i < m} \exists \w_i
  \phi_i(\w_i; \a \x) \land \bigland_{j < n} \neg \exists \z_j
  \psi_j(\z_j; \a \x)
\end{equation} where again $(\a_0, B)$ is a minimal pair for some $a_0
\subseteq A$ and the
$\phi_i, \psi_j$ are $k$-intrinsic formulae. 

We will call a formula $\phi(\x; \a)$ {\it intrinsic} over $\a$ when
it is $k$-intrinsic over $\a$ for some $k \in \omega$

Intrinsic formulae are meant as approximate descriptions of $\cl(\a)$;
we get a full description by taking the collection of all such
approximations.  This leads us to define the closure type of a
finite tuple $\a$ as the set of all intrinsic formulae $\phi(\x; \w)$
for which $\phi(\x; \a)$ is realized.

\begin{defn}
  \def\w{\bar w}
  For any fixed tuple $\a$, the closure-type of $\a$, denoted
  $\cltp(\a)$, is defined by \[
  \begin{aligned}
    \cltp(\a) = &\{\, \exists \x \phi(\x; \w) : \phi
    \text{ is intrinsic over } \w, \C \models \exists \x \phi(\x; \a) \,\}
    \cup \\ &\{\, \neg \exists \y \psi(\y; \w) : \psi
    \text{ is intrinsic over } \w, \C \models \neg \exists \y \psi(\y;
    \a) \,\} 
  \end{aligned}
  \]
  
  We refer to the formulae in the first set as {\it positive} and those
  in the second as {\it negative}.
  
  For fixed $\a$ and $M$ any set, the closure type of $\a$ over $M$
  is defined by $\cltp(\a / M) = \bigcup_{\m \finsubs M} \cltp(\a \m)$.
\end{defn}

We want to show that for tuples $\a$ and $\b$, if
$\cltp(\a / M) = \cltp(\b/M)$, then
$\tp_{\C}(\a / M) = \tp_\C(\b / M)$.  We will do this by showing that
for $a_0$, there is a $b_0$ with $\cltp(\a a_0) = \cltp(\b b_0)$.
Since this implies that $\a a_0$ and $\b b_0$ are partially
isomorphic, this will allow us to establish back-and-forth
equivalence.  The proof will divide into two cases, based on whether
or not $a_0 \in \cl(\a)$.

\begin{lem}\label{cl-lem}
  Let $\a, \b$ be tuples with $\cltp(\a) = \cltp(\b)$.  Then for any
  $a_0 \in \cl(\a)$ there is a $b_0$ with $\cltp(\a a_0) = \cltp(\b b_0)$.
\end{lem}

\begin{proof}

  Suppose first that there is some $A$ with $(\a, A)$ a minimal pair
  and $a_0 \in A$.  Let $p$ be the type 
  \[
  \begin{aligned}
  \{\,  \Delta_A(\b \x)  \,\}
  &\cup \bigcup \{\, \exists \y \phi(\y; \b x_0) : \exists \y \phi(\y;
  \a a_0) \in \cltp(\a a_0)  \,\} \\
  & \cup  \bigcup \{\, \neg \exists \z
  \phi(\z; \b x_0) :  \neg \exists \z \phi(\y;
  \a a_0) \in \cltp(\a a_0)  \,\}
\end{aligned}\]
 where we assume without loss that
  $A$ is enumerated as $a_0 a_1 \ldots a_n$ and $\x$ is $x_0 x_1
  \ldots x_n$.  Note that any finite fragment of $p$ will be implied
  by a single intrinsic formula of the form \[ 
  \Delta_A(\b \x) \land
  \bigland_i \exists \y_i \phi_i(\y_i; \b x_0) \land \bigland_j \neg \exists \z_j
  \phi(\z_j; \b x_0) \]
  Thus $p$ is finitely satisfiable since $\cltp(\b) =
  \cltp(\a)$.  Also, since $\C$ is $\omega$-saturated, $p$ is
  realized by some $b_0 \ldots b_n$.  Then by definition $\cltp(\b b_0)
  = \cltp(\a a_0)$.

  If $a_0$ is not contained in a minimal extension as above, then it
  will be contained in a minimal chain $A_0 \subseteq A_1 \ldots
  \subseteq A_l$ where $(A_i, A_{i+1})$ is a minimal pair.  Thus
  iterating the above argument will suffice for the general case.
\end{proof}

To handle the case where $a_0 \not \in \cl(\a)$, we will want to
employ the following relative notion of closure. This was noted in
\cite{baldwin-shi} and \cite{baldwin-shelah} where it was used to
define semi-genericity.

\begin{defn}
  For $A \subseteq B \subseteq M$, we say that $B$ is closed {\it over
    $A$} (in $M$) if $\cl_M(B) = \cl_M(A) \cup B$.
\end{defn}

The idea is that the any minimal pairs which originate in $B$ actually
originate in $A$, so that extending to $B$ adds no new minimal pairs.

The following lemma was noted by Baldwin and Shi in  \cite{baldwin-shi}

\begin{lem}
  Let $A \leq B \in \K$.  Let $C_1, \ldots, C_m$ be extensions of $B$
  for which $B \not \leq C_i$ but $A \leq A(C_i \setminus B)$.  Then
  for $M \models T_{(\K, \leq)}$, any embedding of $A$
  into $M$ extends to an embedding of $B$ into $M$ which does not
  extend to an embedding of any of the $C_i$.
\end{lem}

Together with the $\omega$-saturation of $\C$, we have 
\begin{cor}\label{full-embed}
  For any $A \finsubs \C$, if $A \leq B$ then $A$ extends to
  an embedding of $B$ which is closed over $A$.
\end{cor}

In order to apply this to the case in which $a_0 \not \in \cl(\a)$, we
show that any finite fragment of $\cltp(\a a_0)$ can be realized by
embedding a finite closed set over $\cl(\b)$.

\begin{lem}
  Suppose $\a, \b$ are finite tuples with $\cltp(\a) = \cltp(\b)$.
  Fix $a_0 \in \C \setminus \cl(\a)$, and let $\Sigma$ be a finite
  fragment of $\cltp(\a a_0)$.  Then there is a {\bf finite} $D
  \subseteq \cl(\a a_0)$ such that if $f$ is an embedding of $D$
  with $f: \a \mapsto \b$ and $f(D)$ closed over $\cl(\b)$,  then
  there some $b_0 \in f(D)$ so that $ \C \models \Sigma(\b b_0)$. 
\end{lem}

\begin{proof}

  \def\c{\bar c} \def\d{\bar d} \def\w{\bar w} \def\v{\bar v}
  \def\x{\bar x} \def\e{\bar e} \def\m{\bar m} 


  We may assume without loss that $\Sigma$ consists of $k$-intrinsic
  formulae for some $k$.  Starting with $\Sigma_0 = \Sigma$ and
  $E_{-1} = \emptyset$, we inductively construct sets of formulae
  $\Sigma_l$ and sets $E_l$ which are realizations of the positive
  part of $\Sigma_l$.  For each positive $\exists \x \pi(\x)$ in
  $\Sigma_l$, we add a tuple $\e_\pi$ to $E_l$ such that
  $\C \models \pi(\e_\pi)$.  We then create $\Sigma_{l+1}$ to ensure
  that $E_k \models \pi(\e_\pi)$.  We do this by adding every positive
  and negative formula in $\pi(\e_\pi)$ to $\Sigma_{l+1}$.  We also
  need to ensure that $E_k \models \neg \exists \x \nu(x)$ when the
  latter is in $\Sigma_l$.  We note that $\neg \exists \x \nu(x)$ will
  be equivalent to
  \[\forall \x \left[\Delta_B(\x \a a_0 \e) \rightarrow \biglor \neg
    \exists \y \phi_\gamma(\y; \x \a a_0 \e) \lor \biglor \exists \w
    \psi_\delta(\w; \x \a a_0 \e) \right] \]
  for some tuple of parameters $\e$..  Let $\b_0, \ldots, \b_l$
  enumerate all realizations of $B$ in $E_l$ (if any).  For each
  $\b_i$ we can choose a formula $\theta_i^j$ which witnesses the
  disjunction (thus, $\theta_i^j$ will be either
  $\neg \exists \y \phi_\gamma$ or $\exists \w \psi_\delta$).  We add
  all such $\theta_i^j$ to $\Sigma_l$ as well.

  Let $D=E_k$ and suppose $f$ is a closed embedding of $D$ which maps
  $\a \mapsto \b$.  Letting $b_0$ denote the image of $a_0$ under this
  embedding, it is straightforward to show that
  $ \C \models \Sigma(\b b_0)$ as desired.

\end{proof}

\begin{lem}\label{main-lem}  
  Suppose $\cltp(\m \a) = \cltp(\m \b)$ for tuples $\m, \a, \b$.  Then
  for any $a_0 \in \C \setminus \cl(\m \a)$ there is a $b_0$ such that
  $\cltp(\m \a a_0) = \cltp(\m \b b_0)$
\end{lem}

\begin{proof}
  Let $\Sigma(\m \a a_0)$ be a finite fragment of $\cltp(\m \a b_0)$.
  By the previous lemma, there is a finite extension $E$ of
  $\m \a a_0$ so that for any $b_0$, if $F$ is a closed embedding of
  $E$ over $\m \b d$, then $\Sigma(\m \b b_0)$ will hold.  Since
  $(\K, \leq)$ is a full amalgamation class, we can find a closed
  embedding of $Ec$ over $\m \b$ by Corollary \ref{full-embed}; we let
  $b_0$ correspond to $a_0$ under such an embedding.  Thus we can
  realize every finite fragment of $\cltp(\m \a a_0)$ while fixing
  $\m \b$, so that by compactness we can find a $b_0$ as required.
\end{proof}

\begin{cor}\label{total-cor}
  Suppose $\a,\b$ are tuples with $\cltp(\a) = \cltp(\b)$.  Then for
  any $c_0 \in \C$, there are tuples $\c, \d$ with $c_0 \in \c$ and
  $\cltp(\a \c) = \cltp(\b \d)$
\end{cor}

\begin{proof}
  This is a combination of Lemmas \ref{cl-lem} and \ref{main-lem}.
\end{proof}

Finally, we prove that elements with the same closure type over a set
have the same type over that set.
\begin{prop}\label{closed-elem}
  Let $M \subseteq \C$ with $|M| < |\C|$.  Let $a, b \in \cl(M)$ and
  suppose $\cltp(a/M) = \cltp(b/M)$.  Then $\tp(a/M) = \tp(b /
  M)$
\end{prop}

\begin{proof}
  We show that for $\m \finsubs M$, $(\C, \m a) \isom_k (\C, \m b)$
  for any finite $k$, where $\isom_k$ denotes $k$-move back-and-forth
  equivalence.  By \ref{total-cor}, for any tuples $\c, \d$ with
  $\cltp(\m a \c ) = \cltp(\m b \d)$ and $c'$, we can find a $d'$ so
  that $\cltp(\m a \c c') = \cltp(\m b \d d')$.  This implies that
  $(\C, \m a \c c') \isom_{k-1} (\C, \m b \d d')$, and establishes the
  back-and-forth equivalence.  By the \fraisse-Hintikka theorem and
  compactness, this establishes that $\tp(a/M) = \tp(b/M)$.
\end{proof}

\section{Classes which Interpret Arithmetic}
\label{sec:arith}
In Section \ref{sec:background} we introduced the classes
$(\K_r^+, \sleq_r)$ for rational $r$.  These classes have intrinsic
transcendentals and essentially undecidable theories
\cite{brody-laskowski}.  We note in this section that it is the
presence of intrinsic transcendentals in conjunction with an ability
to combine such extensions in arbitrary configurations that leads to
the essential undecidability of the resulting theory.  In Section
\ref{sec:acoll} we derive a class from $(\K_r^+, \sleq_r)$ which will
limit the existence of such configurations and produce a simple theory.

Our result here is simply a generalized restatement of the main
theorem of \cite{brody-laskowski}, which asserts that under certain
conditions the theory of the generic will interpret Robinson's $R$ and
hence be essentially undecidable.

\begin{thm}
  Let $(\K, \leq)$ be a full amalgamation class, and suppose there is
  an intrinsically transcendental pair $(A,B)$ for which there are
  $V \in \K$, $U \subseteq B \setminus A$, and $X \in \K$ such that:
  \begin{enumerate}
  \item $U \oplus_V U' \subseteq X$, where $U' \isom U$.
  \item $(U \oplus_V U', X)$ is an intrinsically transcendental
    biminimal pair (that is, it is a minimal pair and there is no $A
    \subsetneq U \oplus_V U'$ for which $(A, X)$ is minimal pair).
  \item For $ n \in \omega$, $\bigoplus_{i < n} (U_i \oplus U_i') \leq
    \bigoplus_{i < n} (X_i/V)$ where $X_i \isom X$
  \end{enumerate}

  Then the theory of the $(\K, \leq)$-generic interprets Robinson's
  $R$ and is hence essentially undecidable.
\end{thm}

The details of the proof are omitted here; it is a straightforward
generalization of the special case proved in \cite{brody-laskowski}.
The main ideas are that each copy of $A$ will code a natural number
via the number of copies of $B$ that appear over it.  Copies of $V$
will be codes for bijections between copies of $B$ over different
copies of $A$, and the structures $\bigoplus_{i < n} (X_i/V)$ will
witness the presence or absence of such a bijection.  The structures
$U$ are simply appropriate substructures of $B$ used to define the
bijection; $V$ and $U$ were simply points in \cite{brody-laskowski}.
Using this coding, one can interpret the graphs of addition and
multiplication on the natural numbers.  It is worth noting that the
presence of the appropriate configurations in \cite{brody-laskowski}
relied on the presence of sufficiently large intrinsically
transcendentals extensions and the neutrality of the minimal pair
relation $(A,B)$ to the internal structure of $A$.

\section{Limiting Intrinsic Types}
\label{sec:limiting}

\newcommand{\fcl}{\mathop{\mathrm{scl}^-_M}}
\newcommand{\gcl}{\mathop{\mathrm{scl}^+_M}}
\newcommand{\nbcl}{\mathop{\mathrm{scl}^-_M}}
\newcommand{\pbcl}{\mathop{\mathrm{scl}^+_M}}
\renewcommand{\S}{\mathop{\mathcal{S}}}

In this section we give a condition on classes with intrinsic
transcendentals that will limit the number of intrinsic types, and in
the next section we will give a non-trivial example of such a class.

Our approach will be to associate intrinsic types with {\it supports};
these will be subsets of a model which determine the type.  The
definition of a support will depend on that of an {\it external
  closure}; this will consist of unions of chains of minimal pairs
with a base in $M$ but all other extensions external to $M$.
\begin{defn}
  Let $(\K, \leq)$ be a full amalgamation class, let $M \models
  T_{(\K, \leq)}$ and let $\C = \C_{(\K, \leq)}$.  For $X \subseteq
  M$, we define the external closure $\ecl_M(X)$ as follows.
  \begin{itemize}
  \item Let $J_0$ be the union of all $B \subseteq \C$ with
    $(X',B)$ a minimal pair for some $X' \finsubs X$, and $B \cap M \subseteq X $.
  \item Given $J_i$, let $J_{i+1}$ be the union of all $B
    \subseteq \C$ with  $(X',B)$ a minimal pair for $X' \finsubs
    J_i$ and $B  \cap M \subseteq X$.
  \end{itemize}
  We then define $\ecl_M(X)$ as $X \cup \bigcup_{i \in \omega}
  J_i$.
\end{defn}

\begin{rem}
 Note that for $a \in \ecl_M(X) \setminus M$, there is a minimal chain $X =
 X_0 \subseteq X_1 \subseteq \ldots \subseteq X_k$ with $a \in X_k$
 and $X_i \cap M = X$ for all $i$.
\end{rem}

\begin{defn} Let $(\K, \leq)$ be a full amalgamation class with $\C =
  \C_{(\K, \leq)}$.  Given a model $M \models T_{(\K, \leq)}$, a {\it
    support system for $\cl(M)$ over $M$} is a set $\S$ of subsets of
  $M$ such that for $a \in \cl(M) \setminus M$, there is some $S \in
  \S$ so that $a \in \ecl(S)$.  Such an $S$ is called a {\it support}
  for $a$ over $M$.

  Moreover, 
  \begin{itemize}
  \item $\S$ is {\it bounded by $\mu$} if every $S \in \S$ has cardinality
    less than $\mu$; 
  \item $\S$ has {\it unique supports} if for $a \in \cl(M) \setminus M$ there is
    exactly one $S \in \S$ with $a \in \ecl(S)$.  In such cases $S$
    may be denoted by $\supp_M(a)$;
  \item $\S$ has {\it free closures} if for $S,T \in \S$ \[
    \ecl_M(ST) =
    \ecl_M(S) \oplus_{ST} \ecl_M(T)   \]
  \item $\S$ is {\it edge-closed} if $\S$ has unique supports and for
    $a \in \cl(M) \setminus M, \m \finsubs M$, if $R(a\m)$ holds for
    some relation  $R$ of the language, then $\m \subseteq \supp_M(a)$.
  \end{itemize}
  
  We note that a support system with free closures will also have
  unique supports.  We say that the class $(\K, \leq)$ has {\it
    limiting supports} if there is a cardinal $\lambda$ such that any
  sufficiently large model has a support system which is bounded by
  $\lambda$, has free closures and is edge-closed.
\end{defn}

The thrust of the notion of limiting supports is that for any model
$M$ and $a \in \cl(M) \setminus M$, any interaction between $a$ and $M$ will
mediated by $\supp_M(a)$.  The support of $a$ will thus determine
$\tp(a/M)$.  Since there will be a uniform bound on the number of such
supports, for sufficiently large models $M$ there will be at most
$|M|$ intrinsic types.

The notion of a limiting support system is motivated by the following example.
\begin{ex}\label{acyclic_graphs}
  The class $(\K_1^+, \sleq_1)$ consists of finite acyclic graphs with
  $A \sleq_1 B$ when the connected components of $B \setminus A$ are
  disjoint from the components of $A$ (i.e.
  $B = (B \setminus A) \oplus A$).  For $M$ a model, any vertex
  $a \in \cl(M) \setminus M$ must be contained in a tree which has its
  root vertex in $M$ and all other vertices in $\ecl(M)$.  The root
  vertices can be taken as supports, so that the set of all single
  vertices of $M$ forms a support system for $M$.  In fact this
  defines a limiting support system.  It is not hard to show that
  there are at most $2^{\aleph_0}$ extrinsic types over any model of
  $T_{(\K_1^+, \sleq_1)}$ so by Corollary \ref{stab-spec} below, the
  theory is strictly superstable.\footnote{A published form (see
    \cite{baldwin-easter}) of a conjecture of Baldwin's stated that no
    generic had a strictly superstable theory.  While this example
    provides a counterexample (see \cite{ikeda2012ab} for another); it is clear from the literature
    (e.g. Example 2.35 of \cite{baldwin-shi}) that Baldwin meant to
    deny the existence of a generic with a {\bf small} strictly
    superstable theory.}
\end{ex}
The above example generalizes slightly to any class of graphs in which for $M
\models T_{(\K, \leq)}$ and $a, b \in M$, there are only finitely many
pairwise disjoint paths from $a$ to $b$ in $M$.  
We also present two examples which do not have limiting supports.
\begin{ex}
  Consider the class $(\K_r^+, \sleq_r)$ for rational
  $r$ with $0 < r < 1$ as discussed above.  For $M$ a model
  of $T_{(\K_r^+, \sleq_r)}$, an obvious choice for a
  support system would be the set of finite subsets of $M$.  Such a
  support system would not have free closures or even unique supports.
  In fact, it is not hard to see that no support system for this class
  with bounded supports could have unique supports.  Choose
  $X, Y \in \K^+_r$ so that there are $W,Z \in \K^+_r$ with
  $X \subseteq W$, $Y \subseteq Z$, $X \cap Y = \emptyset$,
  $W \cap Z \neq \emptyset$, and
  $\delta_r(W/X) = 0 = \delta_r(Z/Y)$.  Then (by
  compactness) we can choose a model $M$ with copies of $X, Z$ in
  different supports with intersecting external closures.
\end{ex}
\begin{ex}\label{supp-ex2}
  \def\f{\bold{f}} Consider the class $(\K_f, \sleq)$ where $\K_f$ is
  the class of all finite graphs, and $A \sleq B$ is defined as
  $A \sleq_1 B$ for $A \neq \emptyset$ (that is $\delta_1(B_0/A) > 0$
  for $A \subsetneq B_0 \subseteq B$) and $\emptyset \sleq B$ for
  every $B \in \K_f$.  Then it is straightforward to show that
  $(\K_f, \sleq)$ satisfies the properties
  \ref{kax}.\ref{kax:1}-\ref{kax}.\ref{kax:8}.  It is not a full
  amalgamation class, since for arbitrary $B, C$ we do not have
  $C \sleq_1 B \oplus_\emptyset C$.  It does however satisfy Corollary
  \ref{full-embed}, hence Lemma \ref{closed-elem} holds in this
  context.

  Note that for any finite $A$, if $B$ is the one point extension of
  $A$ that connects the only vertex of $B \setminus A$ to every vertex
  of $A$ then $(A,B)$ will be a minimal pair.  A simple compactness
  argument shows that for any $\kappa < |\C|$, there is a model
  $M_\kappa$ and an $A_\kappa \subseteq M_\kappa$ with $|A| = \kappa =
  |M|$ and a $b$ such that for $a \in A_\kappa$, $(a,b)$ is an edge.
  Since $\tp(b/A)$ is non-algebraic, it has $|\C|$ realizations and
  hence some realizations in $\C \setminus M_\kappa$.  Let $b_\kappa$
  be such a realization; then for any $A_0 \subseteq A_\kappa$ we have
  $b_\kappa \in \ecl(A_0) $.  Thus $b_\kappa$ has $2^\kappa$ different
  supports.

  We note that no support system could be limiting.  If a support
  system $\S$ is bounded by $\lambda$, then for $\kappa > \lambda$
  there will be $\kappa$ supports for $b_\kappa$ in $A_\kappa$ so that
  the supports would not be unique.  Further the system would not be
  edge-closed.
\end{ex}

In what follows, when a class has limiting supports we will implicitly
choose a support system that witnesses this for whatever model $M$ is under
discussion.  Thus when the next lemma chooses a support for an element
over a model, it should be understood that the support is chosen from
a system with the properties witnessing limiting supports.

To show that classes with limiting supports have few intrinsic types,
we will argue that the type of an element over a model is determined
by its support.  The main part of this is contained in the following
lemma.



\begin{lem} Assume that $(\K, \leq)$ is a full amalgamation class with
  limiting supports.  Let $M \models T_{(\K, \leq)}$, let $X \subseteq M$ be a
  support, and fix tuples $\a,\b \finsubs \C$.  Suppose
  \begin{itemize}
  \item $\tp(\a /  X) = \tp(\b / X)$
  \item For every $a \in \a \setminus M$ and every $b \in \b \setminus
    M$,     $\supp_M(a) = X = \supp_M(b)$.
  \item $\a \cap M = \b \cap M$
  \end{itemize}

  Then for $\m \finsubs M$ and $\c$ with $(\a \m, \c)$ a minimal pair,
  there is a tuple $\d$ with
  \begin{enumerate}
  \item $\a \m \c \isom \b \m \d$ as witnessed by some isomorphism $\alpha$
  \item $\tp(\a \c / X) = \tp( \b \d / X) $
  \item For $c \in \c \setminus M$, $\supp_M(c) = X$ if and only if 
    $\alpha(c) \in \d \setminus M$ and $\supp_M(\alpha(c)) = X$.
  \item $\c \cap M = \d \cap M$
  \end{enumerate}
\end{lem}
\begin{proof}
  
  Let $\c_0 := \{\, c \in \c : c \not \in M, \supp_M(c) = X \,\} $ and
  let $\c_1 = \c \setminus (\c_0 X)$.  Note that for $c \in \c_1$,
  either $\supp(c) \neq X$ or $c \in M$.  In the first case, free
  closures ensures that there is no relation which holds between $c$
  and any element of $\c_0$, while in the latter the same is
  guaranteed by the definition of a support.  Thus, letting
  $\x = \c \cap X$, we have
  $\c = \a \m \x \c_0 \oplus_{\a \m \x} \a \m \x \c_1$.  Since
  $(\a \m, \a \m \c)$ is a minimal pair, we have
  $\a \m \leq \a \m \x \leq \c_0 \oplus_{\a \m \x} \c_1$ if
  $\c_0 \neq \emptyset \neq \c_1$, contradicting the minimality of
  $(\a \m, \a \m \c)$.  Thus $\c = \c_0$ or $\c = \c_1 \x$.

    If $\c = \c_0$, let $q$ be obtained from $tp(\c / \a X)$ by
    mapping $\a$ to $\b$.  Then $q$ is consistent since $\tp(\a / X) =
    \tp( \b / X)$.  By saturation, it has more than $|M|$ realizations
    (since $M$ is algebraically closed).  Choose $\d$ to be a
    realization contained in $\C \setminus M$.
  
    If $\c = \c_1 \x$, then we let $\d = \c$.  We want to show that
    $(\a \m, \c) \isom (\b \m, \c)$.  If $(a, m)$ is an edge with $a
    \in \a \setminus M, m \in \m$, we have that $m \in X$ (since $\S$
    is edge-closed) and similarly for $b \in \b \setminus M$.  Thus we
    have $\a \m \isom \b \m$.  Since no element of $\c_1 \setminus M$
    has support $X$, by free closures there is no relation which holds
    between $\a \setminus M$ and $\c \setminus M$.  Thus $\a \m \c
    \isom \b \m \c$ since $\a \cap M = \b \cap M$.
    
    This establishes the first statement.  In both cases, the
    remaining statements are clear.
\end{proof}

\begin{cor} Assume that $(\K, \leq), M$ and $X$ are  as above.  Fix
  tuples $\a,\b \finsubs \C$.
  Suppose
    \begin{itemize}
    \item $\tp(\a /  X) = \tp(\b / X)$
    \item For $a \in \a \setminus M$ and $b \in \b \setminus M$,
      $\supp(a) = X = \supp(b)$.
    \item $\a \cap M = \b \cap M$
    \end{itemize}  Then
  \begin{enumerate}
  \item\label{cor1} For $\m \finsubs M$ and $\exists \z \phi(\z; \a
    \m) \in \cltp(\a \m)$, we have $\exists \z \phi(\z; \b \m) \in
    \cltp(\b \m)$
  \item\label{cor2} $\cltp(\a/M) =   \cltp(\b / M)$.
  \end{enumerate}
\end{cor}

\begin{proof}

  The second statement is an immediate consequence of the first; we
  prove \eqref{cor1} by induction on $k$, the least number such that
  $\phi$ is $k$-intrinsic.  The case when $k=0$ is immediate from the
  previous lemma.  For the inductive step, fix $\c$ with
  $\phi(\c; \a \m)$.  By the previous lemma, there is some $\d$ with
  $(\b \m, \d) \isom (\a \m, \c)$, $\tp(\a \c / X) = \tp(\b \d / X)$,
  $\c \cap M = \d \cap M$ and every element of $\d \setminus M$ has
  support $X$.  Thus the hypotheses of the corollary are satisfied, so
  by the inductive hypothesis the $(k-1)$-intrinsic formulae realized
  over subsets $\b \d \m$ are precisely those realized over subsets of
  $\a \c \m$.  This establishes that $\phi(\d; \b \m)$ holds.

\end{proof}

\begin{prop}\label{int-types-cond}
  If $(\K, \leq)$ is a full amalgamation class with limiting supports,
  then any sufficiently large model $M$ of the theory of the $(\K,
  \leq)$-generic has $|M|$ intrinsic types.
\end{prop}

\begin{proof}
  Let $\mu$ be a cardinal so that all supports can be chosen to be of
  cardinality less than $\mu$; fix $M$ a model with cardinality $\kappa
  \geq 2^\mu$.  For $a, b \in \cl(M) \setminus M$ with $\supp_M(a) =
  \supp_M(b)$ and $\tp(a/ \supp_M(a)) = \tp(b/ \supp_M(b))$ , the
  preceding corollary implies that $\cltp(a/M) = \cltp(b/M)$ and hence
  $\tp(a/M) = \tp(b/M)$.  Since there are most $\kappa$ choices for
  $\supp(a)$ and $2^\mu$ choices for $\tp(a/X)$, we have
  at most $\kappa \cdot 2^{\mu} = \kappa$ intrinsic types.  
\end{proof}

\begin{cor}\label{stab-spec}
  If $(\K, \leq)$ is a full amalgamation class with limiting supports, then:
  \begin{enumerate}
  \item If for any $M \models T_{(\K, \leq)}$, 
    $|M| = \kappa \geq 2^{\aleph_0}$ implies that $S(M)$ contains has
    at most $\kappa$ extrinsic types, then $T_{{(\K, \leq)}}$ is
    superstable.
  \item If for any $M \models T_{(\K, \leq)}$,
    $|M| = \kappa \geq 2^{\aleph_0}$ implies that $S(M)$ contains has
    at most $\kappa^{\aleph_0}$ extrinsic types,then the theory of the
    generic is stable.
  \end{enumerate}
\end{cor}

\section{The Anti-Collapse}
\label{sec:acoll}

In this section we apply the results of Section \ref{sec:limiting} to
derive a class with few intrinsic types from an arithmetic class.
The generic of this class will have a strictly simple theory.

Fixing $r$ rational from $(0,1)$, we work with the class
$(\K_r^+, \sleq_r)$ discussed in the introduction.  Recall that this
class has intrinsic transcendentals and the generic has an essentially
undecidable theory.  It is straightforward to show that for any $M
\models T_{(\K_r^+, \sleq_r)}$, there are $2^{|M|}$ intrinsic types
over $M$.

We will limit the structure of intrinsically transcendental extensions
to meet the conditions of Proposition \ref{int-types-cond}.  The
procedure is analogous to Hrushovski's collapse (see \cite{wagner}),
but reversed in the following sense.  For $B, M$ any graphs embedded
in a common superstructure, the {\it basis} of $B$ in $M$ is the set
of vertices of $M$ which have an edge to some vertex in $B$.  We
denote it by
$B^M := \{\, m \in M : (m,b) \text{ is an edge for some } b \in B
\,\}$.
Then while the original collapse limited the number of extensions over
a fixed basis with relative predimension $0$, we instead limit the
structure of the possible bases of such extensions.  We thus term our
construction an {\it anti-collapse}.

\begin{notn} We adopt the following notations:
  \begin{itemize}
  \item $\delta$ refers to $\delta_r$.
  \item We extend $\delta$ to all pairs $A,B$ embedded in a common
    superstructure by setting $\delta(B/A) = \delta(AB/A)$.
  \end{itemize}
\end{notn}

The following facts are easily established:
\begin{lem}  For $(\K_r^+, \sleq_r)$ and $\delta$ as above:
  \begin{enumerate}
  \item $(\K_r^+, \sleq_r)$ is a full amalgamation class
  \item  $\delta(\oplus_{i<n} (B_i/A))=\sum_{i<n}
    \delta(B_i/A)$ for $A, B_i \in \K_r^+$
  \item If $A,B,C \in \K_r^+$ are embedded in a common superstructure, then
    $\delta(B/AC)\le \delta(B/A)$
  \item For $A \subseteq B \in \K_r^+$, $\delta(B/A) = |B \setminus A| - r
    e(B,A)$, where $e(B,A)$ denotes the number of edges from $B
    \setminus A$ to $A$
  \end{enumerate}
\end{lem}

Note that intrinsically transcendental extensions will correspond to
intrinsic extensions $B \supsetneq A$ with $\delta(B/A) = 0$.  For any
class, a {\it biminimal pair} $(A,B)$ is a minimal pair for which
there is no $A' \subsetneq A$ with $(A', B)$ a minimal pair.  Our
approach will be to limit the copies of $A$ which can give rise to
copies of $B$ -- this will maintain the presence of intrinsic
transcendentals but force a tree-like structure onto the closure and
limit the number of intrinsic types.

Our basic procedure is to fix a parameter $N \in \omega$ and partition
vertices into equivalence classes of size at most $N$.
Intuitively, a set of vertices $A$ whose elements all belong to the
same class can serve as the basis for an arbitrary number of copies of
an intrinsically transcendental extension $B$, while sets which
contain elements of different classes can only serve as the basis for
finitely many such extensions.

To proceed formally, let $L^* = L_{G} \cup \{\, S(x,y) \,\}$ where
$L_{G}$ is the language of graphs; we will
amalgamate classes of $L^*$-structures.  $S(x,y)$ will be used to
define an equivalence relation on the vertices; a set of vertices
whose elements are all in the same class will be referred to as {\it
  $S$-homogeneous}.

\begin{notn}
  Our derived amalgamation class will be based on the following
  parameters.
\begin{itemize}
\item Choose $N \in \omega $ with $N > 1$.  This will determine the
  size of the equivalence classes under $S$.
\item Let $\gamma$ denote the smallest rational number for which
  $\delta(B/A) \sleq -\gamma$ whenever $\delta(B / A) < 0$ (this is a
  simple case of the {\it granularity} in \cite{mcl}; its existence is
  immediate from the rationality of $r$).
\item We choose a function $\mu(X,Y): \K_r^+ \times \K_r^+ \to \R$ to
  be any function with
  \[\mu(X, Y) \ge \frac{2 \left(\delta(X) + \delta(Y)
    \right)}{\gamma}\]
      
\item Our choice of $\mu$ gives rise to $\nu(X): \K_r^+ \to \R$,
  defined by
    $$\nu(X) = \max \{\, \mu(X_0, X_1) : X = X_0 \sqcup X_1 \,\} $$
    where $X_0 \sqcup X_1$ denotes the disjoint union of $X_0$ and
    $X_1$.
\end{itemize}
\end{notn}

We will need to work with a generalization of biminimal pairs which we
call proper $0$-extensions.  These represent minimal extensions of
relative predimension $0$ which cannot be decomposed into independent
extensions.  We need this generalization because simply limiting the
occurrence of biminimal pairs over non-homogeneous bases would still
allow dependencies between the closures of different homogeneous sets;
this would result in a failure to have free closures on the resulting
support system.

\begin{defn} For $X \subseteq Z \in \K_r^+$, we say that $(X,Z)$ is a
  $0$-{\it extension} if there is a minimal chain
  $X = X_0 \subseteq X_1 \subseteq \ldots \subseteq X_k = Z$ with
  $\delta(X_{i+1}/X_i)= 0$.  We will call the extension {\it proper}
  if there is no $X_0 \subsetneq X$ with $(X_0, Z)$ a $0$-extension
  and there are no non-empty $Z_0, Z_1$ with $Z = XZ_0 \oplus_X XZ_1$.
\end{defn}

We note that if $(X,Z)$ is a $0$-extension, then there is no $Y$
with $X \subseteq Y \subseteq Z$ and $\delta(Y/X) < 0$.  Also note
that biminimal $0$-extensions are examples of proper
$0$-extensions.

We seek to limit the appearance of proper $0$-extensions while
preserving the notion of a closed substructure.  Because biminimal
extensions can be created during amalgamation, we cannot require that
every biminimal extension occur over a base which is $S$-homogeneous.
Our obstruction to doing so, however, is only finite and will not
affect the external closures since models are algebraically closed.

\def\ws{*}

\begin{defn}\label{adm}
  We say that an $\L^*$-structure $A$ is {\it admissible} if it
  satisfies the following conditions.
  \begin{enumerate}
  \item\label{ad1} $A \rest L_G \in \K_r^+$
  \item\label{ad2} $S$ defines an equivalence relation on $A$, with
    classes of cardinality less than $N$.
  \item\label{ad3} For $X \subseteq A$, if $X$ is not
    $S$-homogeneous, then for $(X, Y)$ a proper $0$-extension,
    there are at most $\nu(X)$ pairwise disjoint copies of $Y$ over
    $X$ in $A$.  A copy of $Y$ over $X$ is a $Y'$ so that there is
    some isomorphism $f:  Y \to Y'$ which fixes $X$.  This
    condition will restrict the number of intrinsic types.
  \end{enumerate}

  Let $\K_r^{\ws}$ denote the class of all finite admissible
  $L^*$-structures.    We extend $\delta$ to $\K_r^*$ by defining
  $\delta(A) = \delta(A \rest \L_G)$.
\end{defn}

\begin{defn}\label{ads}
  For $A,B \in \K_r^\ws$, we say that $A \sleq^* B$ if
  \begin{enumerate}
  \item\label{ads:1}  $A \subseteq B$ 
  \item\label{ads:2} If $X \subseteq B$ is $S$-homogeneous, then $X
    \subseteq A$ or $X \subseteq B \setminus A$.
  \item\label{ads:3} $\left(A\! \rest\! L_G \right) \sleq \left(B
      \!\rest\! L_G \right)$
  \end{enumerate}
\end{defn}

We want to show that $(\K_r^\ws, \sleq^*)$ is a full amalgamation
class.  The proof relies on Hrushovski's insight in the collapse that
the number of new biminimal pairs which can be created during free
amalgamation is bounded (see \cite{wagner}, pp. 174-5).  The
proposition below and its proof are modifications of his argument (as
presented by Wagner).

\begin{notn}
  If $D = B \oplus_A C$, then for $X \subseteq D$:
  \begin{itemize}
  \item $X_B := X \cap \left(B \setminus A\right)$
  \item $X_A := X \cap A$
  \item $X_{BC} := X \cap \left( BC \setminus A \right) $
  \end{itemize}
\end{notn}

\begin{prop}\label{bounded1}
  Fix $A,B,C \in \K_r^*$ with $A = B \cap C, A \sleq^* B, A \sleq^* C$ and let
  $D = B \oplus_A C$.  For $U \subseteq D$ with $U_B \neq \emptyset
  \neq U_C$, there are at most $\nu(U)$ pairwise disjoint proper
  $0$-extensions of $U$ in $D$
\end{prop}

\begin{proof}
  It suffices to show that for $U = X \sqcup X'$ with
  $X_B \neq \emptyset \neq X'_C$, if $\{\, (XX', Z_i) : i < k \,\} $
  is a sequence of proper $0$-extensions which are pairwise disjoint
  over $XX'$, then $k < \mu(X, X')$.  For each $i < k$, we want to
  replace $X, X'$ with $Y_i, Y_i'$ so that $(Y_i Y_i', Z_i)$ is a
  proper $0$-extension which is minimal in the sense that there are no
  ${\hat Y_i}, {\hat Y_i'}$ extensions of $Y_i, Y_i'$ for which
  $({\hat Y_i} {\hat Y_i'}, Z_i)$ is a proper $0$-extension.  We
  choose $Y_i, Y_i' \subseteq Z_i$ to be maximal extensions of $X$ and
  $X'$ so that $\delta(Y_i / X) = 0 = \delta(Y_i' / X')$ and
  $Y_i, Y_i'$ are disjoint over $X, X'$.

  Letting $W_i = Z_i \setminus Y_iY_i'$, the following are straightforward
  \begin{itemize}
  \item  $W_i \neq \emptyset$
  \item  $\delta(W_i / Y_iY_i') = 0$ 
  \item  $W_i^{Y_i} \neq \emptyset \neq  W_i^{Y_i'}$
  \item  $Y_i^{XX'} \subseteq X$, $Y_i'^{XX'} \subseteq X'$
  \end{itemize}

  \bigskip {\bf Case 1.}  Renumbering as needed, let $Z_1 \ldots Z_m$
  be such that there is $V_i \subseteq
  Z_i$ with $(Y_iY_i', V_i)$ a minimal $0$-extension and $(V_i)_B \neq
  \emptyset \neq (V_i)_C$.  We will show that $\delta( (V_i)_{BC} / A
  (Y_i Y_i')_{BC}) < 0$; since $A \sleq^* BC$ this will limit the number
  of possible such $V_i$ and hence limit $m$.

  It is easily established that $(V_i)_A \neq \emptyset$ and
  $V_i^{Y_i} \neq \emptyset \neq V_i^{Y_i'}$ (the latter fact comes
  from our choice of $Y_i, Y_i'$).  By minimality,
  $\delta((V_i)_{BC} / Y_iY_i'(V_i)_A) < 0$ so $\delta((V_i)_{BC} / A
  (Y_i Y_i')_{BC}) < 0$ as well.

  Let $YY' := \bigcup_{i=1}^m Y_i Y_i'$ and let $Q = \bigcup_{i=1}^m
  (V_i)_{BC}$.  Then note that \[
  \begin{aligned}
    \delta( Q / A (YY')_{BC}) & \leq \sum_i
    \delta\Bigl((V_i)_{BC} / A (YY')_{BC}\Bigr)\\ &\le \sum_i \delta\Bigl((V_i)_{BC}
    / A (Y_i Y_i')_{BC}\Bigr) \\
    &\leq -\gamma m
  \end{aligned}
  \]    
  Since $A \sleq^* D$ we have
  $$
  \begin{aligned}
    0 &< \delta\left( (YY')_{BC} Q/ A\right)\\
    &=  \delta\left( Q / A (Y Y')_{BC}\right) + \delta\left((YY')_{BC} / A\right) \\
    & \leq - \gamma m   + \delta((Y Y')_{BC} / A)\\
    &\leq - \gamma m  +  \delta((Y Y')_{BC} / (XX')_A (YY')_A) \\
    &= - \gamma m + \delta(XX') - \delta(X_A X_A' Y_A Y_A') \\
    &\leq -\gamma m + \delta(X) + \delta(X')
  \end{aligned}
  $$ where the final equality holds since $\delta(YY' / XX') = 0$.  
  Thus $m < \frac{\delta(X) + \delta(X')}{\gamma}$

  \def\sleqp{\sleq^*} \def\ya{Z_A}
  \def\yb{Z_B} \def\yc{Z_C}
  
  \bigskip
  {\bf Case 2.}  
  We show that $k - m \leq \frac{\delta(X) + \delta(X')}{\gamma}$.
  For $i > m$, every minimal $0$-extension $(Y_i Y_i', V_i)$ satisfies
  $(V_i)_B$ or $(V_i)_C$ empty.  Let $Z_{m+1}, \ldots, Z_l$ denote
  those copies with $(V_i)_B = \emptyset$.  Since $W_i^{Y_i} \neq
  \emptyset$, we can choose $V_i$ so that $(V_i)^{Y_i} \neq
  \emptyset$.  Thus we have $\delta(Y_i / V_i X_B X_A) < \delta(Y_i /
  X_B X_A) = \delta(Y_i / X) = 0$.  Letting $Y^* = \bigcup_{i=m+1}^l
  Y_i$ and letting $V^* = \bigcup_{i=m+1}^l V_i$ we have the following
  inequalities:
  $$
  \begin{aligned}
    \delta(X_B Y^* / AC)  - \delta(X_B Y^* / X_A) &\leq \delta(X_B Y^* / V^* X_A)  - \delta(X_B Y^* / X_A) \\
    &= \delta(Y^* / V^*  X) + \delta(X_B / V^* X_A)  - \delta(X_B / X_A) \\
    &\leq \delta(Y^* / V^* X) + \delta(X_B / X_A) - \delta(X_B /
    X_A) \\
    &\leq \sum_{i = m+1}^l \delta(Y_i / V^* X)  \\
    &\leq -\gamma(l-m)
  \end{aligned}
  $$  
  Since $AC \sleq^* ABC$, we have $\delta(X_B Y^* / AC) > 0$.  Thus
  $$
  \begin{aligned}
    l-m &\leq \frac{\delta(X_B Y^* / X_A) - \delta(X_B Y^* /AC) }{\gamma}
    \\
    &\leq \frac{\delta(X_B Y^* / X_A) }{\gamma} \\
    &\leq \frac{ \delta(X)}{\gamma}
  \end{aligned}
  $$   The same reasoning shows that $k - l \leq
  \frac{\delta(X')}{\gamma}$.  

  \bigskip
  
  We thus have $k \leq \mu(X,X')$ as desired.
\end{proof}

\begin{prop}
  $(\K_r^*, \sleq^*)$ is a full amalgamation class which satisfies
  \ref{kax}.\ref{kax:1} through \ref{kax}.\ref{kax:8}.
\end{prop}

\begin{proof}
  We need to show that $(\K_r^*, \sleq^*)$ is closed under free
  amalgamation and has full amalgamation in addition to satisfying the
  axioms of \ref{kax}.  

  We first show that $(\K_r, \sleq^*)$ is closed under free
  amalgamation.  Fix $A,B,C \in \K_r^*$ with $A = B \cap C$, $A \sleq^*
  B, A \sleq^* C$.  Letting $D = B \oplus_A C$, we first have to show that
  $D \in \K_r^*$ and that $B \sleq^* D$, $C \sleq^* D$.

  Of the properties defining admissibility, only \ref{ads}.\ref{ads:3}
  needs comment (\ref{adm}.\ref{ad2} holds by \ref{ads}.\ref{ads:2}).
  Fix $X' \subseteq D$ with $X$ not $S$-homogeneous, and let
  $(X, Y)$ be a proper $0$-extension.  Since $B,C \in \K_r^*$, we may
  assume without loss that $X_B$ and $X_C$ are both non-empty.
  Then Proposition \ref{bounded1} implies that there are most
  $\nu(X)$ copies of $Y$ over  $X$ in $D$.  That $B \sleq^* D$ and $C
  \sleq^* D$ are straightforward

  For full amalgamation, we have to show that if $A, B,C,D$ are as
  above with the exception that we only require $A \subseteq C$, then
  $C \sleq^* D$.  We know that $C \rest L_G \sleq_r D \rest L_G$ since
  $(\K_r^+, \sleq_r)$ is a full amalgamation class.  Thus it suffices
  to show that for $X \subseteq C$, if $X$ is not $S$-homogeneous and
  $(X, Y)$ is a proper $0$-extension, then there are fewer than
  $\nu(X)$ copies of $Y$ over $X$ in $D$.  Since $A \sleq^* B$, any
  proper $0$-extension of $X$ must be contained in $C$.  Since
  $C \in \K_r^*$, we have $C \sleq^* D$ as required.

  The verifcations of  \ref{kax}.\ref{kax:1} through
  \ref{kax}.\ref{kax:8} are routine. 
\end{proof}

We now show that the generic has few intrinsic types by applying the
results of Section \ref{sec:limiting}.  Again, we let $\C$ be a universal domain for
the theory of the $(\K_r^*, \sleq^*)$-generic.  

Note that for $X \subseteq M$ with $M \models T_{(\K_r^*,
  \sleq^*)}$, if $(X,Y)$ is a minimal pair where $Y$ enlarges an
$S$-class of $X$, then by the second criterion for admissibility, $Y$
is algebraic over $X$ and hence contained in $M$.  Thus the only
intrinsic extensions of $X$ which can be external to a model will be
intrinsic in the sense of $(\K_r^+, \sleq_r)$.

\begin{lem}\label{ecl-dim}
  Let $M \models T_{(\K_r^*, \sleq^*)}$ and fix $X \finsubs M$.  For
  $X = X_0 \subseteq X_1 \subseteq \ldots \subseteq X_k$ a minimal
  chain over $X$ with $X_i \not \subseteq M$ for $i > 0$,
  $\delta(X_k / X) = 0 = \delta(X_k / M)$
\end{lem}

\begin{proof}
  Note that if $\delta(X_k / X) < 0$ then $X_k$ is an
  algebraic extension and we have $X_k \subseteq M$.  Since $0 \leq
  \delta(X_k / M) \leq \delta(X_k / X)$, $\delta(X_k / M) = 0$ as
  well.  
\end{proof}

\begin{lem}\label{beta2}
  Let $M \models T_{(\K_r^*, \sleq^*)}$ and let $X \subseteq M$.  Then
  $$\ecl_M(X) = \bigcup \{\,  \ecl_M(X'):  X' \subseteq X,\  X' \text{
    $S$-homogeneous} \,\} $$
\end{lem}

\begin{proof}
  One direction of the equality is obvious, so it suffices to show
  that if $a \in \ecl_M(X)$, then for some $X' \subseteq X$ with $X'$
  S-homogeneous, $a \in \ecl_M(X')$.  Let $X = X_0 \subseteq X_1
  \subseteq \ldots \subseteq X_k$ be a minimal chain with $a \in
  X_k$.  We have to show that $X_i \cap M = X_0$.

  By Lemma \ref{ecl-dim}, $(X_i, X_{i+1})$ is a $0$-extension for $i
  \geq 0$.  Letting $X'' = X_1^M$, we have that $(X'', X_1)$ is a
  biminimal $0$-extension. By property \ref{ads}.\ref{ads:3}, if $X''$
  is not $S$-homogeneous then $X_1$ is an algebraic extension of $X''$,
  contradicting that $X_1 \not \subseteq M$.  Letting $X'$ be the
  $S$-closure of $X''$, I claim that for $i > 1$ and $Y' = X_i^M$, $Y'
  \subseteq X'$.  If not, then $X', Y'$ are in different $S$-classes,
  and it is easy to see that $X_i$ is a proper $0$ extension of a non
  $S$-homegenous subset of $X'Y'$.  Thus
  by \ref{adm}.\ref{ad3}, there can only be finitely many copies of
  $X_i$ over $X'Y'$, contradicting the algebraic closure of $M$.
\end{proof}

\begin{prop}\label{hassupp}
  For $a \in \cl(M) \setminus M$, there is $X \finsubs M$ with $X$
  $S$-homogeneous and $a \in \ecl_M(X)$.
\end{prop}
\begin{proof}
  Let $a \in \cl(M) \setminus M$.  Fix $X \finsubs M$ with $a \in
  \cl(X)$.  We show that $a \in \ecl_M(X)$.  Let $X = X_0 \subseteq
  X_1 \subseteq \ldots \subseteq X_k$ be a minimal $0$-chain with $a \in
  X_k$.  Without loss, we may assume that $X_i \not \subseteq M$ for
  $i > 0$ (otherwise we can replace $X$ with $\bigcup X_j$ for $X_j
  \subseteq M$.)

  Suppose, by way of contradiction, that $a \not \in \ecl_M(X)$.  Then
  choose the least value $j$ so that $X_{j+1} \cap M \neq X$.  Let $W
  = (X_{j+1} \setminus X_j) \cap M, V = (X_{j+1} \setminus X_j)
  \setminus M$.  Then $\delta(WV/X_j) = 0$, so $\delta(V/X_j W) < 0$
  (by minimality).  Note that $\delta(X_j V / XW) = \delta(V / X_j X
  W) + \delta(X_j / XW) \leq \delta(V / X_j W) + \delta(X_j / X) <
  0$.  Thus $X_j V \subseteq M$, a contradiction which establishes
  that $a \in \ecl_M(X)$.  We can choose $X$ $S$-homogeneous and
  $S$-closed by Lemma \ref{beta2}.
\end{proof}

\begin{thm}
  $(\K_r^*, \sleq^*)$ has limiting supports.
\end{thm}
\begin{proof}
  For a fixed model $M$, let $\T$ be
  $\{\, X \subseteq M : X \text{ is } S\text{-homogeneous}\,\}$.  Then
  $\T$ is clearly a support system bounded by $\omega$ which is
  edge-closed.  We show that $\T$ has free closures.  We need to show
  that for $X, Y$ supports,
     \[\label{suppeq1}\ecl_M(X) \ecl_M(Y) = \ecl_M(X) \bigoplus_{XY}
     \ecl_M(Y)\]
     Suppose that $X \neq Y$ are $S$-homogeneous subsets of $M$, $X = X_0 \subseteq X_1 \subseteq \ldots \subseteq X_k$ and
    $Y = Y_0 \subseteq Y_1 \subseteq \ldots \subseteq Y_l$ are minimal
    $0$-chains with $X_i, Y_i \not \subseteq M$ for $i > 0$.  Then if
    $(X_k \setminus  Y_{k-1})  \cap  (Y_l \setminus  Y_{l-1}) \neq
    \emptyset$, $X_k Y_l$ would contain a proper $0$-extension over
    some non-homogeneous subset of  $XY$, contradicting that there can
    be only finitely many such.
\end{proof}

\bigskip

We now show that the theory of the generic is strictly simple.

\begin{prop} For any rational $r \in (0,1)$, there is an $N_r$ so that
  for $N > N_r$, the theory of the $(\K_r^\ws, \sleq^*)$ generic has
  the independence property.
\end{prop}
\begin{proof}
  \def\b{\bar b} 

  Lemma 3.2 of \cite{brody-laskowski} shows that for sufficiently
  large $N$ we can find a biminimal $0$-extension $(A,C)$ from $\K_r$
  with $|A| = N$.  By the definition of admissibility, we can expand
  $A$ to an $L^*$-structure so that for some $a \in A$,
  $A \setminus \{\, a \,\} $ is $S$-homogeneous.  Let $\phi(x, \b)$ be
  the $L^*$ formula stating that $x\b$ has the quantifier-free type of
  $A$ and $x \b$ extends to a copy of $C$.  Fix
  $M \models T_{(\K^*, \sleq^*)}$, and let $\b_1, \ldots, \b_n$ be
  $L^*$-isomorphic to $\b$.  By biminimality, $\b \sleq a C_{\b}$ for
  $C_{\b}$ a copy of $C$ over $a\b$.  For $\eta \in 2^n$, let
  $D_\eta = \bigoplus_{\eta(i)=1} (C_{\b_i} / a)$.  Then by Lemma 2.5,
  there is a strong embedding of $D_\eta$ over $M$, so that
  $\models \phi(a, \b_i)$ exactly when $\eta(i)=1$.
\end{proof}

This highlights that simply limiting the number of intrinsic types
over a model is not sufficient to guarantee the stability of the
theory.  We do have a simple theory, however.

\begin{thm}\label{simple}
  The theory of the $(\K_r^\ws, \sleq^*)$-generic is simple.
\end{thm}

\begin{proof}
  Employing Theorem 2.8 of \cite{casanovas} we show that for
  $\kappa, \lambda$ infinite cardinals, if $W$ is any set of pairwise
  incompatible types of size at most $\lambda$ over a parameter set
  $A$ of size $\kappa$, then $|W| \leq \kappa^\omega + 2^\lambda$.
  Since we have already shown that there are at most $\kappa$
  intrinsic types over $A$, we may assume that $W$ consists of
  extrinsic types.

  Let us call two types $p,q$ over $A$ {\it closure-type incompatible}
  if there is some finite $\a \finsubs A$ and an intrinsic formula
  $\phi(x; \a)$ such that $p \vdash \phi$ and $q \vdash \neg \phi$.
  Imitating 3.3 in \cite{casanovas} we will show that if $X$ is a set
  of pairwise mutually closure-type incompatible types over $A$, with
  each $p \in X$ of power at most $\lambda$, then
  $|X| \leq 2^\lambda$.  Indeed, we can uniquely associate each
  $p \in X$ with a function $f_p$ from $A^{< \omega}$ to the set of
  intrinsic types over $\emptyset$ by
  $f_p(\a) = \{\, \phi(\x; \y) : p(x) \vdash \phi(x; \a) \,\}$.  Thus
  we can map $X$ to an anti-chain from
  $\fn_{\lambda^+}(\kappa, 2^{\aleph_0})$, the set of all partial maps
  from a cardinality $\lambda$ subset of $\kappa$ to a subset of
  $ 2^{\aleph_0}$.  Lemma IV.7.5 of \cite{kunen} then implies
  $|X| \leq \left( 2^{\aleph_0}\right)^\lambda = 2^\lambda$.

  We note that if $p,q$ are incompatible types, then one of the
  following must hold ( $a,b$ are chosen so that $a \models p$ and
  $b \models q$)
  \begin{enumerate}
  \item\label{simp1} $p$ and $q$ are closure-type incompatible over $A$
  \item\label{simp2} The relative predimension of $\cl(Aa)$ over $Aa$ is negative,
    as is the relative predimension of $\cl(Ab)$ over $Ab$, so that
    any realization of $p \cup q$ would have to be intrinsic over $A$.
    Specifically, $\delta( \cl(Aa) / Aa) < 0, \delta( \cl(Ab) / Ab) <
    0$ and $\delta(
    \cl(Aa) / Aa) + \delta( \cl(Ab) / Ab) \leq -1$.
  \item\label{simp3} The number of edges in from $a$ to $A$ and $b$ to
    $A$ would make any
    realization of $p \cup q$ intrinsic over $A$.  Specifically, this
    occurs when $1 - r|a^A \cup b^A| \leq 0$.
  \end{enumerate}

  For $p \in W$, let $A_p \finsubs A$ be a minimal subset of $A$ such
  that for $a \models p$, $a^A \subseteq A_p$ and every minimal chain
  over $Aa$ with negative relative dimension over $Aa$ will be over
  $A_p a$; let $F$ denote the map $p \mapsto A_p$.  Note that for
  $p,q \in W$, if $A_p = A_q$ then $p$ and $q$ must be incompatible
  for reason \eqref{simp1} (in this case, reason \eqref{simp2} reduces to
  \eqref{simp1}); thus for a fixed $A_p$ there are at most
  $ 2^\lambda$ types  $q \in W$ with $F(q) = A_p$.  Since there
  are at most $\kappa$ possible $A_p$, there are at most $\kappa
  2^\lambda = \kappa + 2^\lambda$ types in $W$.

\end{proof}

\bibliographystyle{amsplain}

\bibliography{biblio}

\providecommand{\bysame}{\leavevmode\hbox to3em{\hrulefill}\thinspace}
\providecommand{\MR}{\relax\ifhmode\unskip\space\fi MR }
\providecommand{\MRhref}[2]{%
  \href{http://www.ams.org/mathscinet-getitem?mr=#1}{#2}
}
\providecommand{\href}[2]{#2}
\begin{thebibliography}{10}

\bibitem{baldwin-easter}
John~T Baldwin, \emph{Problems on pathological structures}, Proceedings of 10th
  Easter conference in model theory, Wendisches Rietz, Citeseer, 1993,
  pp.~1--9.

\bibitem{baldwin-shelah}
John~T. Baldwin and Saharon Shelah, \emph{Randomness and semigenericity},
  Trans. Amer. Math. Soc. \textbf{349} (1997), no.~4, 1359--1376. \MR{MR1407480
  (97j:03065)}

\bibitem{baldwin-shi}
John~T. Baldwin and Niandong Shi, \emph{Stable generic structures}, Ann. Pure
  Appl. Logic \textbf{79} (1996), no.~1, 1--35. \MR{MR1390325 (97c:03103)}

\bibitem{brody-laskowski}
J.~Brody and M.C. Laskowski, \emph{On rational limits of shelah--spencer
  graphs}, Journal of Symbolic Logic \textbf{77} (2012), no.~2, 580--592.

\bibitem{casanovas}
Enrique Casanovas, \emph{The number of types in simple theories}, Annals of
  Pure and Applied Logic \textbf{98} (1999), no.~1, 69--86.

\bibitem{Hodges}
W.~Hodges, \emph{{Model Theory}}, Cambridge University Press, 1993.

\bibitem{hrushovski-stable-pp}
Ehud Hrushovski, \emph{{A stable $\aleph_0$-categorical pseudoplane}}, preprint
  \textbf{249} (1988).

\bibitem{hru-acf}
\bysame, \emph{Strongly minimal expansions of algebraically closed fields},
  Israel Journal of Mathematics \textbf{79} (1992), no.~2-3, 129--151.

\bibitem{hrushovski-new-sm}
\bysame, \emph{A new strongly minimal set}, Ann. Pure Appl. Logic \textbf{62}
  (1993), no.~2, 147--166, Stability in model theory, III (Trento, 1991).
  \MR{MR1226304 (94d:03064)}

\bibitem{ikeda2012ab}
Koichiro Ikeda, \emph{Ab initio generic structures which are superstable but
  not $\omega$-stable}, Archive for Mathematical Logic \textbf{51} (2012),
  no.~1-2, 203--211.

\bibitem{mcl-kuek}
D.~W. Kueker and M.~C. Laskowski, \emph{On generic structures}, Notre Dame J.
  Formal Logic \textbf{33} (1992), no.~2, 175--183. \MR{MR1167973 (93k:03032)}

\bibitem{kunen}
Kenneth Kunen, \emph{Set theory}, Studies in Logic, vol.~34, College
  Publications, 2013.

\bibitem{mcl}
Michael~C. Laskowski, \emph{A simpler axiomatization of the {S}helah-{S}pencer
  almost sure theories}, Israel J. Math. \textbf{161} (2007), 157--186.
  \MR{MR2350161}

\bibitem{pourmahdian}
Massoud Pourmahdian, \emph{Smooth classes without ac and robinson theories},
  Journal of Symbolic Logic (2002), 1274--1294.

\bibitem{wagner}
Frank~O. Wagner, \emph{Relational structures and dimensions}, Automorphisms of
  first-order structures, Oxford Sci. Publ., Oxford Univ. Press, New York,
  1994, pp.~153--180. \MR{MR1325473}

\end{thebibliography}

\end{document}